\documentclass[reqno]{amsart}
\usepackage[english]{babel}
\usepackage[utf8]{inputenc}
\usepackage[T1]{fontenc}
\usepackage[dvips]{graphicx}
\usepackage{amsfonts,amssymb,amsthm,enumerate,newlfont,amscd,amsmath}
\usepackage{mathrsfs}  
\usepackage{latexsym}
\usepackage{color}
\usepackage{hyperref}

\newtheorem {teo} {\bf Theorem\,} [section]
\newtheorem {prop} [teo] {\bf Proposition}
\newtheorem {coro} [teo]{\bf Corollary}
\newtheorem {lemma} [teo] {\bf Lemma}
\newtheorem {defn} [teo] {\bf Definition}

\newtheorem {Remark} [teo] {\bf Remark}

\newcommand{\be}{\begin{eqnarray}}
\newcommand{\ee}{\end{eqnarray}}
\newcommand{\benn}{\begin{eqnarray*}}
\newcommand{\eenn}{\end{eqnarray*}}
\newcommand{\bse}{\begin{equation}}
\newcommand{\ese}{\end{equation}}
\newcommand{\bsenn}{\begin{displaymath}}
\newcommand{\esenn}{\end{displaymath}}
\newcommand{\sgn}{\textrm{sgn}}

\baselineskip
\date{\today}
\title[Property for non local evolution equations]
{Dissipative property for non local evolution equations}

\author[S. H. da Silva\\ 
]
{Severino H. da Silva} 

\address{Severino Horácio da Silva \newline
Unidade Acadêmica de Matemática UAMat/CCT/UFCG\\
Rua Aprígio Veloso, 882,  Bairro Universitário, CEP.: 58429-900, 
Campina Grande - PB, Brazil}
\email{horaciousp@gmail.com; horacio@dme.ufcg.edu.br}

\author[A. R. G. Garcia 
]
{Antonio R. G. Garcia} 

\address{Antonio Ronaldo Gomes Garcia \newline
Centro de Ciências Exatas e Naturais,
Universidade Federal Rural do Semi-Árido,
Mossoró-RN, Brazil,
Av. Francisco Mota, 572,  CEP.: 59.625-900.}
\email{gomesgarcia@gmail.com; ronaldogarcia@ufersa.edu.br}

\author[B. E. P. Lucena 
]
{Bruna E. P. Lucena} 

\address{Bruna E. Pereira Lucena\newline
Unidade Acadêmica de Matemática UAMat/CCT/UFCG\\
Rua Aprígio Veloso, 882,  Bairro Universitário, CEP.: 58429-900, 
Campina Grande - PB, Brazil}
\email{brunapereiraufcg@gmail.com }

\thanks{Partially supported by CAPES-Brazil}
\subjclass[2000]{45J05, 45M05, 37B25}
\keywords{Non local equation; Well-posedness; Smoothness orbit; Global Attractor; Lyapunov's functional}

\begin{document}
\maketitle

\begin{abstract}
In this work we consider the non local evolution problem
\[
\begin{cases}
\partial_t u(x,t)=-u(x,t)+g(\beta K(f\circ u)(x,t)+\beta h), ~x
\in\Omega, ~t\in[0,\infty[;\\
u(x,t)=0, ~x\in\mathbb{R}^N\setminus\Omega, ~t\in[0,\infty[;\\
u(x,0)=u_0(x),~x\in\mathbb{R}^N,
\end{cases}
\]
where $\Omega$ is a smooth bounded domain in $\mathbb{R}^N; ~g,f:
\mathbb{R}\to\mathbb{R}$ satisfying certain growing condition and $K$ 
is an integral operator with  symmetric kernel,
$ Kv(x)=\int_{\mathbb{R}^{N}}J(x,y)v(y)dy.$ 
We prove that
Cauchy problem above is well posed, the solutions are smooth with
respect to initial conditions, and we show the existence of a global 
attractor. Futhermore, we exhibit a Lyapunov's functional, concluding that the 
flow generated by this equation has a gradient
property. 
\end{abstract}



\section{Introduction}\label{intro}
We consider the non local evolution problem
\begin{equation}\label{prob}
\begin{cases}
\partial_t u(x,t)=-u(x,t)+g(\beta K(f\circ u)(x,t)+\beta h), x\in\Omega,~t\in[0,\infty[\\
u(x,t)=0, ~x\in\mathbb{R}^N\setminus\Omega, ~t\in[0,\infty[\\
u(x,0)=u_0(x),~x\in\mathbb{R}^N,
\end{cases}, 
\end{equation}
where $u(x,t)$ is a real function on $\mathbb{R}^{N} \times [0,\infty[$, $\Omega$ is a bounded smooth domain in $\mathbb{R}^N ~(N\geq 1);
~h$ and $\beta$ are nonnegative constants;
$f,g:\mathbb{R}\to\mathbb{R}$ are locally Lipschitz continuous 
satisfying some growth conditions and $K$ is an integral operator 
with symmetric nonnegative kernel, given by
\begin{equation}\label{mapK}
Kv(x):=\intop_{{\mathbb{R}^{N}}} J(x,y)v(y)dy,
\end{equation}
where $J$ is a symmetric non negative function of class $\mathscr{C}^{1}$, with  

$$\intop_{\mathbb{R}^N} J(x,y)dy=\intop_{\mathbb{R}^N}J(x,y)dx=1.$$ 

The dynamics of non local evolution
Equations like in (\ref{prob}) has attracted the attention of many 
researchers in the last years; see for instance
\cite{Amari,Amari2, FAH, Chen,Coombes,DaSilvaPereira, Silva,Silva2,Silva3,Ermentrout,
Ermentrout3,Ermentrout2,Kishimoto,Kubota,Laing,Rubin,Wilson}
and \cite{Wu}. However, the model considered here presents innovation
and generalizes the model considered  \cite{FAH,DaSilvaPereira, Laing} and
\cite{Orland}, which can be obtained as a particular case of
(\ref{prob}) with $f$ being the identity, as well as it  generalizes 
the model considered in
\cite{Kishimoto,Laing,Rubin,Silva,Silva2,Silva3,Wilson} and \cite{Wu}, 
which can be obtained as a particular case of (\ref{prob}) 
where $g$ is the identity, $\beta=1$ and the integral operator $K$ is
the convolution product. When $g$ and $f$ are identity, $\beta=1$ and 
the integral operator $K$ is the convolution product, we also obtain as particular case of (\ref{prob}) the
model considered in \cite{Cortazar}.  

The approache considered here  was
motivated by similar approaches in   \cite{FAH, DaSB} and
\cite{rossi}, whose basic idea is to find an 
abstract way to impose Dirichlet boundary conditions in 
non local evolution equations.

The paper is organized as follows. In Section \ref{wellposed}, 
assuming a growth condition on the functions $g$ and $f$,
we prove that (\ref{prob}) is well posed with globally defined
solution. In Section \ref{Smoothness} we prove  that (\ref{prob})   
generates a $\mathscr{C}^{1}$ flow in a space $X$ which is isometric to
$L^{p}(\Omega)$. In Section \ref{global}, we prove existence 
of a global attractor, and establish some regularity properties for it. 
In Section \ref{comparison}, we prove comparison and boundedness 
results for the solutions of (\ref{prob}). 
Finally, in Section \ref{lyap}, we exhibit a continuous Lyapunov's
functional for the flow generated by (\ref{prob}), and we use 
it to prove that the this flow has the gradient property in the sense
of \cite{Hale}.

\section{Well posedness}\label{wellposed}

In this section, we prove that the Cauchy problem (\ref{prob}) 
is well posed in the suitable phase space
\[X=\left\{u\in L^{p}\left(\mathbb{R}^{N}\right)~:~ u(x)=0, 
~\textrm{if}~x\in\mathbb{R}^N\setminus\Omega\right\}\]
with the induced norm of $L^{p}\left(\mathbb{R}^{N}\right)$. 
For this we assume that the functions $g$ and $f$ satisfy 
the ``suitable'' following growth conditions:
{\it there
exist non negative constants $k_1$, $k_2$, $c_1$  and $c_2$ such that
\begin{equation}\label{dissipg}
|g(x)|\leq k_{1}|x|+k_{2},~ \forall ~ x\in\mathbb{R}
\end{equation} 
and
\begin{equation}\label{dissipf}
|f(x)|\leq c_{1}|x|+c_{2}, ~\forall ~ x\in\mathbb{R}.
\end{equation} 
}
The space $X$ is canonically isomorphic to
$L^p(\Omega)$ and we usually identify the two spaces, 
without further comment. We also use the same notation for a
function in $\mathbb{R}^N $ and its restriction to
$\Omega$ for simplicity, wherever we believe the intention 
is clear from the context.

In order to obtain well posedness of (\ref{prob}), 
we  consider the Cauchy problem
\begin{equation}\label{CP}
\begin{cases}
\displaystyle{\partial_t u}=-u+F(u), \\
u(t_0)=u_{0},
\end{cases}
\end{equation}
where the map $F:X\to X$ is defined by
\begin{equation} \label{mapF}
F(u)(x)=
\begin{cases}
g(\beta K(f\circ u)(x)+\beta h), & x\in \Omega,\\
0, & x\in\mathbb{R}^N\setminus\Omega.
 \end{cases}
\end{equation}

Depending on the properties assumed for $J$, the map given by 
(\ref{mapK}) is well defined as a bounded linear operator in 
various functions spaces and, in particular, it will be well 
defined in $X$.

To prove that $F$ given in (\ref{mapF}) is well defined, under the 
conditions given in (\ref{dissipg}) and (\ref{dissipf}), 
we need of the estimates below for the map $K$, which has been proven in
\cite{DaSilvaPereira}.

\begin{lemma}\label{boundK}
Let $K$ be the map defined by (\ref{mapK}) and $\|J\|_{r}$:=
$\sup_{x\in \Omega}\|J(x,\cdot)\|_{L^{r}(\Omega) },~ 1\leq r\leq \infty.$
If $u\in L^p{(\Omega)}, ~1\leq p \leq \infty$, then
$(Ku\in L^{\infty}{(\Omega)}$,

\begin{equation}\label{estimateLq}
 |Ku(x)|\leq\|J\|_{q}\|u\|_{L^p{(\Omega)}}, ~\forall ~ x\in\Omega,
 \end{equation}
where $1\leq q\leq \infty$ is the conjugate exponent of $p$, and
\begin{equation}\label{estimateL1}
\|Ku\|_{L^{p}(\Omega)}\leq\|J\|_{1}
\|u\|_{L^{p}(\Omega)} \leq\|u\|_{L^{p}(\Omega)}.
\end{equation}
Moreover, if  $u\in L^1{(\Omega)}$, then
$Ku \in L^{p}{(\Omega)}, ~ 1\leq p\leq\infty$, and
\begin{equation}\label{estimateLp}
\|Ku\|_{L^{p}(\Omega)}\leq\|J\|_{p}\|u\|_{L^{1}(\Omega)}.
\end{equation}
\end{lemma}

\begin{defn}\label{loclip}
If $E$ is a normed space, we say that a function  
$F:E\to E$ is locally Lipschitz continuous (or simply locally Lipschitz) if,
for any $x_0\in E$, there exists a constant
$C$ and a rectangle $R=\{x\in E~:~ \|x-x_0\|<b\}$ such that,
if $x$ and $y$ belong to $R$, then $\|F(x)-F(y)\|\leq C\|x-y\|$; 
we say that $F$ is Lipschitz continuous on bounded sets if the rectangle $R$ in the
previous definition may chosen as any bounded rectangle in $E$.
\end{defn}

\begin{Remark}
The two definitions in (\ref{loclip})  are equivalent 
if the normed space $E$ is locally compact.
\end{Remark}

\begin{prop} \label{limitationF}
In addition to the hypotheses from Lemma \ref{boundK}, suppose that the
functions $g$ and $f$ satisfy the two growth conditions
(\ref{dissipg}) and (\ref{dissipf}). Then the function $F$ given
by \eqref{mapF} is well defined in $L^{p}(\Omega)$. 
\end{prop}

\begin{proof}
Consider $1\leq p<\infty$ and let $u\in L^{p}(\Omega)$. 
Then, using Hölder inequality (see \cite{folland}) and 
(\ref{dissipf}), we obtain
\begin{equation}\label{estfL1}
\|f(u)\|_{L^1(\Omega)}\leq\intop_{\Omega}[c_1|u(x)|+c_2]dx \leq
c_1|\Omega|^{\frac{1}{q}}\|u\|_{L^p(\Omega)}+c_2|\Omega|,
\end{equation}
where $q$ denotes the conjugate expoent of $p$.

From estimates \eqref{estimateLp} and (\ref{estfL1}), it follows that
\begin{eqnarray}\label{estKfLp}
\|Kf(u)\|_{L^p(\Omega)}&\leq&\|J\|_{p}\|f(u)\|_{L^1(\Omega)}\nonumber\\
&\leq&\|J\|_{ p}(c_1|\Omega|^{\frac{1}{q}}\|u\|_{L^p(\Omega)}+c_2|\Omega|)\nonumber\\
&=&c_1\|J\|_{ p}|\Omega|^{\frac{1}{q}}\|u\|_{L^p(\Omega)}+\|J\|_{ p}c_2|\Omega|.
\end{eqnarray}
Thus, using (\ref{estKfLp}), it follows that
\begin{eqnarray}\label{estFLp}
\|F(u)\|_{L^p(\Omega)}&=&\|g(\beta|Kf(u)|+\beta h)\|_{L^p(\Omega)}\nonumber\\
&\leq&\left(\intop_{\Omega}[\beta k_1|K((f(u))(x)|+k_1\beta h+k_2]^{p}dx\right)^{\frac{1}{p}}\nonumber\\
&\leq&\|\beta k_1|Kf(u)|+(k_1\beta h+k_2)\|_{ L^p(\Omega)}\nonumber\\
&\leq&\beta k_1\|Kf(u)\|_{ L^p(\Omega)}+\|k_1\beta h+k_2\|_{L^p(\Omega)}\nonumber\\
&\leq&\beta k_1(c_1\|J\|_{p}|\Omega|^{\frac{1}{q}}\|u\|_{ L^p(\Omega)}+\|J\|_{ p}c_2|\Omega|)+(k_1\beta h+k_2)|\Omega|^{\frac{1}{p}}\nonumber\\
&=&\beta k_1c_1\|J\|_{p}|\Omega|^{\frac{1}{q}}\|u\|_{L^p(\Omega)}+\beta k_1\|J\|_{ p}c_2|\Omega|+(k_1\beta h+k_2)|\Omega|^{\frac{1}{p}},
\end{eqnarray}
showing that, in this case, $F$ is well defined. 

The proof for $p=\infty$ is straightforward, because 
if $u\in L^{\infty}(\Omega)$, from (\ref{dissipf}) it follows 
that $f(u)\in L^{\infty}(\Omega)$ and, consequently
\begin{eqnarray*}
|K(f(u)(x))|\leq\|J\|_{1}\|f(u)\|_{\infty}=\|f(u)\|_{\infty}.
\end{eqnarray*}
Thus, using (\ref{dissipf}), we obtain
\begin{eqnarray*}
\|Kf(u)\|_{L^\infty(\Omega)}\leq c_1\|u\|_{\infty}+c_2.
\end{eqnarray*}
Hence, from (\ref{dissipg}), we have
\begin{eqnarray*}
\|F(u)\|_{L^\infty(\Omega)}&\leq&k_1\beta\|Kf(u)\|_{L^\infty(\Omega)}+k_1\beta h+k_2\\ 
&\leq&\beta k_1(c_1\|u\|_{\infty}+c_2)+k_1\beta h+k_2.
\end{eqnarray*}
Thus, we conclude the result.
\end{proof}

\begin{prop}\label{wellp}
Suppose, in addition to the hypotheses from Proposition
\ref{limitationF}, that the functions $g$ and $f$ are Lipschitz continuous on bounded. Then the function $F$ given by 
\eqref{mapF} is Lipschitz continuous on  bounded sets of 
$L^{p}(\Omega), ~1\leq p\leq \infty$.
\end{prop}

\begin{proof}
Suppose $1\leq p< \infty$ and let $u,v\in L^{p}(\Omega)$ 
be such that $\|u\|_{L^{p}(\Omega)}\leq r$ and
$\|v\|_{L^{p}(\Omega)}\leq r$. Then $\|u\|_{\infty} \leq r|\Omega|^{-\frac{1}{p}}$ and $\|v\|_{\infty} \leq r|\Omega|^{-\frac{1}{p}}$.
Let $M$ be the Lipschitz constant of $f$ in the interval $[-r|\Omega|^{-\frac{1}{p}}, ~r|\Omega|^{-\frac{1}{p}}]$. Then, for all $x\in\Omega$,
\begin{eqnarray*}
|f(u(x))-f(v(x))|\leq M|u(x)-v(x)|.
\end{eqnarray*}
From (\ref{estimateL1}) it follows that
\begin{eqnarray*}
\|Kf(u)-Kf(v)\|_{ L^{p}(\Omega) }&\leq&\|J\|_{1}\|f(u)-f(v)\|_{L^{p}(\Omega)}\\
&=&\left(\intop_{\Omega}|f(u(x))-f(v(x))|^{p}dx\right)^{\frac{1}{p}}\\
&\leq&\left(\intop_{\Omega} M^{p}|u(x)-v(x)|^{p}dx\right)^{\frac{1}{p}}\\
&=& M\|u-v\|_{L^p(\Omega)}.
\end{eqnarray*}
Now, if 
$
l=c_1\|J\|_{p}|\Omega|^{\frac{1}{q}}r+\|J\|_{ p}c_2|\Omega|
$
and $N$ denotes the Lipschitz constant of $g$ in the interval 
$[-l,l]\subset\mathbb{R}$, using \eqref{estimateLq}, we have that
\begin{eqnarray*}
\|F(u)-F(v)\|_{L^{p}(\Omega)}&\leq& N\beta\|K[f(u)-f(v)]\|_{ L^{p}(\Omega) }\\
&\leq& N\beta\|J\|_{1}\|K[f(u)-f(v)]\|_{L^{p}(\Omega)}\\
&\leq& N\beta M\|u-v\|_{L^p(\Omega)},
\end{eqnarray*}
showing that $F$ is Lipschitz in bounded sets of
${L^p(\Omega)}$ as claimed.
If $p=1$, the proof is similar, but simpler.
Suppose, finally, that $\|u\|_{L^{\infty}(\Omega)}\leq r, ~\|v\|_{L^{\infty}(\Omega)} \leq r$, 
let $M$ be the Lipschitz constant of $f$ and $N$ denotes the Lipschitz
constant of $g$ in the interval $[-l,l]\subset\mathbb{R}$, where now
$l=c_1\|J\|_{q}|\Omega|^{\frac{1}{q}}r+\|J\|_{ p}c_2|\Omega|$. 

Then, using \eqref{estimateLq}, we obtain
\[\|Kf(u)\|_{L^{\infty}(\Omega)}\leq N\|J\|_{1}\|f(u)\|_{\infty}=\|f(u)\|_{\infty}.\]
Whence, we obtain
\[\|F(u)-F(v)\|_{L^{\infty}(\Omega)}\leq N\beta M\|u-v\|_{\infty}.\]
\end{proof}

From Proposition \ref{wellp}, it follows from well known results, 
on ordinary differential equation in Banach 
space, that the problem (\ref{CP}) has a local solution for arbitrary 
initial condition in $X$. For the global existence, we need the following result
(\cite{ladas} - Theorem 5.6.1).

\begin{teo}\label{ladas} Let $X$ be a Banach space, and suppose that
$g:[t_0,\infty[\times X\to X$ is continuous and
$\|g(t,u)\|\leq h(t,\|u\|); ~\forall ~(t,u)\in
[t_0,\infty[\times X$, where $h:[t_0,\infty[
\times\mathbb{R}^{+}\to\mathbb{R}^{+}$ is continuous and
$h(t,r)$ is non decreasing in $r\geq 0$, for each
$t\in[t_0,\infty[$. Then, if the maximal solution
$r(t,t_0,r_0)$ of the scalar initial value problem
\[r'= h(t,r), ~r(t_0)=r_ 0,\]
exists throughout $[t_0, \infty[$, the maximal interval of 
existence of any solution $u(t,t_0,u_0)$ of the initial 
value problem \[\frac{du}{dt}=g(t,u), ~t \geq t_0, ~u(t_0)=u_0,\]
with $\|u_0\|\leq r_0$, also contains $[t_0, \infty[$.
\end{teo}

\begin{coro}\label{globalexist}
Suppose, the same hypotheses from Proposition \ref{wellp}.
Then the problem (\ref{CP}) has a unique globally defined
solution for arbitrary initial condition in $X$, which is  given, 
for $t\geq t_0$, by the ``variation of constants formula''
\begin{eqnarray}\label{EP_1}
u(t,x)=
\begin{cases}
e^{-(t-t_0)}u_{0}(x)+\displaystyle\intop_{t_0}^{t}e^{-(t-s)}g(\beta Kf(u(s,\cdot))(x)+\beta h)ds,& x\in\Omega,\\
0, & x\in \mathbb{R}^N \setminus\Omega.
\end{cases}
\end{eqnarray}
\end{coro}

\begin{proof}   
From Proposition \ref{wellp}, it follows that the right-hand-side of
(\ref{CP}) is Lipschitz continuous in bounded sets of
$X$ and, therefore, the Cauchy
problem (\ref{CP}) is well posed in $X$, with a unique local solution $u(t,x)$,
given by (\ref{EP_1}) (see \cite{DK}).

If $1\leq p<\infty$, from (\ref{estFLp}), we obtain  
that the right-hand-side of (\ref{CP}) satisfies 
\[\|-u+F(u)\|_{L^{p}(\Omega)}\leq (1+\beta
k_1c_1\|J\|_{p}|\Omega|^{\frac{1}{q}})\|u\|_{L^p(\Omega)}
+\beta k_1\|J\|_{ p}c_2|\Omega|+(k_1\beta h+k_2)|\Omega|^{\frac{1}{p}}.\]

If $p=\infty$, we have that the right-hand-side of (\ref{CP}) satisfies 
\[\|-u+F(u)\|_{L^{p}(\Omega)}\leq\beta(1+ k_1\beta c_1)\|u\|_{\infty}+k_1(\beta c_2+\beta h)+k_2.\]

Hence, defining $h:[t_0,\infty[\times\mathbb{R}^{+}\to\mathbb{R}^{+}$, by 
$$
h(t,r)=(1+\beta k_1 c_1\|J\|_{p}|\Omega|^{\frac{1}{q}})r+\beta k_1\|J\|_{ p}c_2|\Omega|+(k_1\beta h+k_2)|\Omega|^{\frac{1}{p}},
$$ 
if $1\leq p< \infty$ or by 
$$
h(t,r)=\beta(1+k_1\beta c_1)r+k_1(\beta c_2+\beta h)+k_2,
$$ 
in the case $p=\infty$, it follows that  (\ref{CP})
satisfies the hypotheses from Theorem \ref{ladas} and the global existence
follows immediately. The variation of constants formula may be
verified by direct derivation. 
\end{proof}

\section{Smoothness of the solutions}\label{Smoothness}

In this section, in addition the hypotheses from previous section, we assume that
the functions $g,f\in \mathscr{C}^{1}(\mathbb{R})$, and $g'$ and $f'$ are
locally Lipschitz and there exist non negative constants 
$k_{3},~k_{4},~c_3$ and $c_4$, such that
\begin{equation}\label{dissipglinha}
|g'(x)|\leq k_{3}|x|+k_{4}, ~\forall,~ x\in\mathbb{R},
\end{equation}
\begin{equation}\label{dissipflinha}
|f'(x)|\leq c_{3}|x|+c_{4}, ~\forall,~ x\in\mathbb{R}.
\end{equation}

The following result has been proven in \cite{Rall}.

\begin{prop}\label{PropRall}
Let $X$ and $Y$ be normed linear spaces, $F:X\to Y$ a map
and suppose that the Gateaux's derivative of $F, DF:X\to
{\cal L}(X,Y)$ exists and is continuous at $x\in X$. Then the
Fréchet's derivative $F'$ of $F$ exists and is continuous at $x$.
\end{prop}

Using Proposition \ref{PropRall}, we have the following result:

\begin{prop}\label{C1flow}
Suppose, in addition to the hypotheses of Corollary \ref{globalexist}
that the function $g$ and $f$ have derivative satisfying (\ref{dissipglinha}) and
(\ref{dissipflinha}), respectively. Then $F$ is 
continuously Fréchet differentiable on $X$ 
with derivative given by
\begin{eqnarray*}
DF(u)v(x):=
\begin{cases}
-v(x)+g'(\beta Kf(u)(x)+\beta h)\beta K f'(u(x))v(x), & x\in\Omega,\\
 0, & x\in\mathbb{R}^N\setminus\Omega.
 \end{cases}
\end{eqnarray*}
\end{prop}

\begin{proof}
From a simple computation, using the fact that $f$ is continuously
differentiable on $\mathbb{R}$, it follows that the Gateaux's
derivative of $F$ is given by
\begin{eqnarray*}
DF(u)v(x):=
\begin{cases}
-v(x)+g'(\beta Kf(u)(x)+\beta h)\beta K f'(u(x))v(x), & x\in\Omega,\\
 0, & x\in\mathbb{R}^N\setminus\Omega.
 \end{cases}
\end{eqnarray*}

The operator $DF(u)$ is clearly a linear operator in $X$.

Suppose $1\leq p<\infty $ and $q$ is the conjugate exponent of $p$.
Then, if $u\in L^p(\Omega)$, using \eqref{dissipglinha} and \eqref{estimateLq}, we obtain
\begin{eqnarray*} 
&&\|g'(\beta Kf(u)+\beta h)\beta Kf'(u)v\|_{ L^p(\Omega) }\leq
\left\{\intop_{\Omega}|g'(\beta K(f(u)(x))+\beta h)\beta K(f'(u(x)))v(x)|^{p} dx\right\}^{\frac{1}{p}}\nonumber\\
&&\leq \bigg\{\intop_{\Omega} \bigg[ k_3\beta|K(f(u)(x))| +k_3\beta h+k_4\bigg]^{p}\beta^{p}|K(f'(u(x)))v(x)|^{p}dx\bigg\}^{\frac{1}{p}}\nonumber\\
&&\leq \bigg\{\intop_{\Omega}[k_3\beta\|J\|_{q}\|f(u)\|_{L^{p}(\Omega)}+k_3\beta
       h+k_4]^{p} \beta^{p}[\|J\|_{q}\|f'(u)\|_{L^{p}(\Omega)}|v(x)|^{p}dx\Bigg\}^{\frac{1}{p}}.\nonumber\\
\end{eqnarray*}
Thus, from \eqref{dissipf} and \eqref{dissipflinha}, we have
\begin{eqnarray}\label{estglinhap}
&&\|g'(\beta Kf(u)+\beta h)\beta Kf'(u)v\|_{L^p(\Omega)}\leq \nonumber  \\
&&\leq \Bigg\{\intop_{\Omega}[k_3\beta\|J\|_{q}(c_1\|u\|_{L^{p}(\Omega)}+c_2|\Omega|^{\frac{1}{p}})+k_3\beta h+k_4]^{p}
\beta^{p}[\|J\|_{q}(c_3\|u\|_{L^{p}(\Omega)}+c_4|\Omega|^{\frac{1}{p}})|v(x)|^{p}dx\Bigg\}^{\frac{1}{p}}\nonumber  \\
&&=(k_3\beta\|J\|_{q}(c_1\|u\|_{L^{p}(\Omega)}+c_2|\Omega|^{\frac{1}{p}})+k_3\beta
    h+k_4)    \beta\|J\|_{q}(c_3\|u\|_{L^{p}(\Omega)}+c_4|\Omega|^{\frac{1}{p}})\|v\|_{L^{p}(\Omega)}.
\end{eqnarray}
From (\ref{estglinhap}), we have
\[\|DF(u) v\|_{L^p(\Omega)}=\left(k_3\beta
\|J\|_{q}\left(c_1\|u\|_{L^{p}(\Omega)}+c_2|\Omega|^{\frac{1}{p}}\right)
+k_3\beta h+k_4 \right)\beta\|J\|_{q}\left(c_3\|u\|_{L^{p}(\Omega)}+c_4|\Omega|^{\frac{1}{p}}\right)\|v\|_{L^{p}(\Omega)},\]
showing that $DF(u)$ is a bounded operator.
In the case $p=\infty$, we have that  
\begin{eqnarray*}\label{estglinhainfinito}
\|DF(u)v\|_{L^{\infty}(\Omega)}&=&\|g'(\beta Kf(u) +\beta h)\beta Kf'(u)v\|_{\infty}\\
&\leq&(k_3\beta\|Kf(u)\|_{\infty}+k_3\beta h+k_4)\beta\|K\circ (f'(u))\|_{\infty}\|v\|_{\infty}\\
&\leq&(k_3\beta\|J\|_{1}(c_1\|u\|_{L^{\infty}(\Omega)}+c_2)+k_3\beta h+k_4)\beta\|J\|_{1}(c_3\|u\|_{L^{\infty}(\Omega)}+c_4)\|v\|_{\infty}\\
&\leq&(k_3\beta (c_1\|u\|_{L^{\infty}(\Omega)}+c_2)+k_3\beta
       h+k_4)\beta (c_3\|u\|_{L^{\infty}(\Omega)}+c_4)\|v\|_{\infty}
\end{eqnarray*}
showing the boundeness of $DF(u)$ also in this case.

Suppose now that $u_1, u_2$ and $v$ belong to $L^p(\Omega), ~1\leq p<\infty$. Then
\begin{eqnarray*}
&&\|(DF(u_1)-DF(u_2))v\|_{L^p(\Omega)}=\\
&&=\|g'(\beta Kf(u_1)+\beta h)\beta Kf'(u_1)v-g'(\beta Kf(u_2)+\beta h)\beta Kf'(u_2)v\|_{L^p(\Omega)}\\
&&\leq I+II,
 \end{eqnarray*}
where
\[I=\|[g'(\beta Kf(u_1)+\beta h)-g'(\beta Kf(u_2)+\beta h)]\beta Kf'(u_1) v\|_{L^p(\Omega)}\]
and
\[II=\|g'(\beta Kf(u_2)+\beta h)\beta K([f'(u_1)-f'(u_2)])v\|_{L^p(\Omega)}.\]
	
Fixed $u_1\in L^{p}(\Omega)$ and letting $u_2\rightarrow u_1$ in
$L^{p}(\Omega)$ follows that 
$\beta Kf(u_2)+\beta h$ is in a ball of $L^{\infty}$ centered   
$\beta Kf(u_1)+\beta h$. Then, since $g'$ is locally
Lipschitz, there exists $C>0$, such that
\begin{eqnarray*}
|g'(\beta Kf(u_1)+\beta h)(x)-g'(\beta Kf(u_2)+\beta h)(x)|&\leq& C\beta|K[f(u_1)-f(u_2)](x)|\\
&\leq& C\beta\|J\|_{q}\|u_1-u_2\|_{L^p(\Omega)}.
\end{eqnarray*}
Thus, using (\ref{estimateLq}), we have that
\begin{eqnarray*}
I &\leq&\left(\intop_{\Omega}|(C\beta\|J\|_{q}\|u_1-u_2\|_{L^p(\Omega)})^{p}\beta^{p}|Kf'(u_1)(x)|^{p}|v(x)|^{p}\right)^{\frac{1}{p}}\\
&\leq&C\beta\|J\|_{q}\|u_1-u_2\|_{L^p(\Omega)}\beta\left(\intop_{\Omega}|Kf'(u_1)(x)|^{p}|v(x)|^{p}\right)^{\frac{1}{p}}\\
&\leq&C\beta^{2}\|J\|_{q}\|u_1-u_2\|_{L^p(\Omega)}\left(\intop_{\Omega}[\|J\|_{q}\|f'(u_1)\|_{L^{p}(\Omega)}]^{p}|v(x)|^{p}\right)^{\frac{1}{p}}.
\end{eqnarray*}
But, from (\ref{dissipflinha}) follows that
\[\|f'(u_1)\|_{L^{p}(\Omega)}\leq c_3\|u_1\|_{L^{p}(\Omega)}+c_4|\Omega|^{\frac{1}{p}}.\]
Hence,
\begin{eqnarray}\label{estimateI}
I&\leq&C\beta^{2}\|J\|_{q}\|u_1-u_2\|_{L^p(\Omega)}\|J\|_{q}(c_3\|u_1\|_{L^{p}(\Omega)}+c_4|\Omega|^{\frac{1}{p}})\|v\|_{L^{p}(\Omega)}.
\end{eqnarray}

Now, using (\ref{dissipglinha}) and (\ref{estimateLq}), we obtain
\begin{eqnarray*}
|g'(\beta Kf(u_2)(x))+\beta h)|&\leq& k_3\beta|Kf(u_2(x))|+k_3\beta h+k_4\\
&\leq& k_3\beta\|J\|_{q}\|f(u_2)\|_{L^p(\Omega)}+k_3\beta h+k_4\\
&\leq& k_3\beta\|J\|_{q}\left(c_1\|u_2\|_{L^p(\Omega)}+c_2|\Omega|^{\frac{1}{p}}\right)+k_3\beta h+k_4.
\end{eqnarray*}
Whence we obtain
\begin{eqnarray*}
II&\leq&[k_3\beta\|J\|_{q}(c_1\|u_2\|_{L^p(\Omega)}+c_2|\Omega|^{\frac{1}{p}})+k_3\beta h+k_4]\beta\|K[f'(u_1)-f'(u_2)]\|_{L^{p}(\Omega)}.
\end{eqnarray*}

Using (\ref{estimateLp}) and Hölder inequality, we have
\begin{eqnarray}\label{estimateII}
II&\leq&\left[k_3\beta\|J\|_{q}\left(c_1\|u_2\|_{L^p(\Omega)}+c_2|\Omega|^{\frac{1}{p}}\right)+k_3\beta h+k_4\right]\beta\|J\|_{p} 
\|[f'(u_1)-f'(u_2)]v\|_{L^{1}(\Omega)}\\
&\leq&\left[k_3\beta\|J\|_{q}\left(c_1\|u_2\|_{L^p(\Omega)}+c_2|\Omega|^{\frac{1}{p}}\right)+k_3\beta h+k_4\right]\beta\|J\|_{p} 
\|[f'(u_1)-f'(u_2)]v\|_{L^{q}(\Omega)}\|v\|_{L^{p}(\Omega)}\nonumber.
\end{eqnarray}
From \eqref{estimateI} and \eqref{estimateII}, follow that 
\begin{eqnarray*}
&&\|[DF(u_1)-DF(u_2)]v\|_{L^p(\Omega)} \leq\\
&&\leq c\beta^{2}\|J\|_{q}\|u_1-u_2\|_{L^p(\Omega)}\|J\|_{q}
\left(c_3\|u_1\|_{L^{p}(\Omega)}+c_4|\Omega|^{\frac{1}{p}}\right)\|v\|_{L^{p}(\Omega)}\\ 
&&+\left[k_3\beta\|J\|_{q}(c_1\|u_2\|_{L^p(\Omega)}+c_2|\Omega|^{\frac{1}{p}})+k_3\beta h+k_4\right]\beta\|J\|_{p} 
\|f'(u_1) - f'(u_2)v\|_{L^{q}(\Omega)}\|v\|_{L^{p}(\Omega)}.
 \end{eqnarray*}
Thus, to prove continuity of the derivative, we only have to show that
$$\|f'(u_1)-f'(u_2)\|_{L^q(\Omega)}\to 0$$  
when
$$\|u_1-u_2\|_{L^p(\Omega)}\to 0.$$
But, from the growth condition on $f'$ it follows that
$$|f'(u_1)(x)-f'(u_2)(x)|^q\leq[c_3(|u_1(x)|+|u_2(x)|)+2c_4]^q$$ 
and a simple computation show that the right-hand is in
$L^{1}(\Omega)$. Then  the result follows from Lebesgue's Convergence Theorem.

In the case $p=\infty$, from \eqref{estimateL1}, we obtain 
\begin{eqnarray*}
&&\|[DF(u_1)-DF(u_2)]v\|_{L^{\infty}(\Omega)}\leq  \\
&&\leq c\beta\|K[f'(u_1)-f'(u_2)]\|_{L^{\infty}(\Omega)}\beta\|Kf'(u_1)v\|_{\infty} \\
 &&+(k_3\beta\|Kf(u_2)\|_{\infty}+k_3\beta h+k_4)\beta\|K[f'(u_1) - f'(u_2)]\|_{L^{\infty}(\Omega)}\|v\|_{L^{\infty}(\Omega)}\\
&&\leq c\beta\|J\|_{1}\|f'(u_1)-f'(u_2)\|_{L^{\infty}(\Omega)}\beta\|J\|_{1}\|f'(u_1)\|_{\infty}\|v\|_{\infty}\\
&&+(k_3\beta\|J\|_{1}\|f(u_2)\|_{\infty}+k_3\beta h+    k_4)\beta\|J\|_{1}\|f'(u_1) - f'(u_2)\|_{L^{\infty}(\Omega)}\|v\|_{L^{\infty}(\Omega)}\\
&&\leq  c\beta\|f'(u_1)-f'(u_2)\|_{L^{\infty}(\Omega)}\beta(c_3\|u\|_{L^{\infty}(\Omega)}+c_4)\|v\|_{\infty}\\
&&+(k_3\beta (c_1\|u\|_{L^{\infty}(\Omega)}+c_2)+k_3\beta
   h+k_4)\beta\|f'(u_1) - f'(u_2)\|_{L^{\infty}(\Omega)}
\|v\|_{L^{\infty}(\Omega)}.
\end{eqnarray*}
And the continuity of $DF$ follows from the continuity of  $f'$.
Therefore, it follows from Proposition \ref{PropRall} that $F$ is Fréchet
differentiable with continuous derivative in $L^{p}(\Omega)$.
\end{proof}

\begin{Remark}
From Proposition \ref{C1flow}, it follows that the flow generated   
by (\ref{CP}), given by $(T(t)u_0)(x)=u(x,t)$, where $u(x,t)$ is 
given in (\ref{EP_1}), is $\mathscr{C}^1$ with respect to initial condition 
(see \cite{Henry}).
\end{Remark}

\section{Existence of a global attractor}\label{global}

We prove, in this section, the existence of a global maximal
invariant compact set ${\cal A}\subset X \equiv L^{p}(\Omega)$ for the flow
of (\ref{CP}), which attracts each bounded set of $X$
(the global attractor, see \cite{Hale} and \cite{Teman}).

We recall that a set ${\cal B} \subset X$ is an
absorbing set for the flow $T(t)$ if, for any bounded set $C
\subset X$, there is a $t_{1}>0$ such that $T(t)C
\subset {\cal B}$ for any $t\geq t_{1}$.

The following result was proven in \cite{Teman}.

\begin{teo}\label{TeoremaTeman}
Let $X$ be a Banach space and $T(t)$ a semigroup on $X$. Assume
that, for every $t, ~T(t)=T_1(t)+T_2(t)$, where the operators
$T_1(\cdot)$ are uniformly compact for $t$ sufficiently
large, that is, for
every bounded set $B$ there exists $t_{0}$, which may depend on
$B$, such that
$$\bigcup_{t \geq t_{0}} T_{1}(t)B$$
is relatively compact in $X$ and $T_2(t)$ is a continuous mapping
from $X$ into itself such that the following holds: For every
bounded set $C\subset X$,
$$r_c(t)=\sup_{\varphi\in C}\| T_2(t) \varphi\|_X\rightarrow 0\quad
\mbox{as} \quad t\rightarrow\infty.$$  
Assume also that there exists an open set ${\cal U}$ and bounded subset $\cal B$ of
$\cal U$ such that $\cal B$ is absorbing in $\cal U$. Then the
$\omega$-limit set of $\cal B, ~{\cal A}=\omega(\cal B)$, is a
compact attractor which attracts the bounded sets of $\cal U$. It
is the maximal bounded attractor in $\cal U$ (for the inclusion
relation). Furthermore, if $\cal U$ is convex and connected, then
$\cal A$ is connected. 
\end{teo}

\begin{lemma}\label{absorbing}
Assume that (\ref{dissipg}) and (\ref{dissipf}) hold with $k_1\beta
c_1<1$. Then, any positive number $\sigma$,  
the ball of radius $$R=(1+\sigma)\left(\frac{k_1\beta c_2+k_1\beta
  h+k_2}{1-k_1 \beta
c_1}\right)$$ is an absorbing set for the flow $T(t)$
generated by (\ref{CP}).
\end{lemma}

\begin{proof} 
If  $u(\cdot,t)$ is a solution of (\ref{CP}) with initial
condition $u(\cdot,0)$ then, for $1\leq p < \infty$, 
\begin{eqnarray*}
\frac{d}{dt}\intop_{\Omega}|u(x,t)|^pdx&=&\intop_{\Omega}p|u(x,t)|^{p-1}\sgn[u(x,t)]u_{t}(x,t)dx\\
&=&-p\intop_{\Omega}|u(x,t)|^{p}dx+p\intop_{\Omega}|u(x,t)|^{p-1}\sgn[u(x,t)]g(\beta Kf(u(x,t))+\beta h)dx.
\end{eqnarray*}
But, using Hölder inequality, (\ref{dissipg}) and (\ref{dissipf}), it follows that
\begin{eqnarray*}
&&\intop_{\Omega}|u(x,t)|^{p-1}\sgn[u(x,t)]g(\beta Kf(u(x,t))+\beta h)dx\leq\\
&&\leq\left(\intop_{\Omega}(|u(x,t)|^{p-1})^{q}dx\right)^{\frac{1}{q}}\left(\intop_{\Omega}|g(\beta Kf(u(x,t))+\beta h)|^{p}dx\right)^{\frac{1}{p}}\\
&&\leq\left(\intop_{\Omega}|u(x,t)|^{p}dx\right)^{\frac{1}{q}}\left(\intop_{\Omega}(k_1|\beta Kf(u(x,t))+\beta h|+k_2)^{p}dx\right)^{\frac{1}{p}}\\
&&\leq\|u(\cdot,t)\|_{L^{p}(\Omega)}^{p-1}\left(k_1\beta\|K(f(u(\cdot,t)))\|_{L^{p}(\Omega)}+\|k_1\beta h+k_2\|_{L^{p}(\Omega)}\right)\\
&&\leq \|u(\cdot,t)\|_{L^{p}(\Omega)}^{p-1} \left( k_1 \beta \|J\|_{1} \|f(u(\cdot,t))\|_{L^{p}(\Omega)} + ( k_1 \beta h +k_2)|\Omega|^{\frac{1}{p}}\right)\\
&&\leq \|u(\cdot,t)\|_{L^{p}(\Omega)}^{p-1} \left( k_1 \beta \left( c_1\|u(\cdot,t)\|_{L^{p}(\Omega)} + c_2 |\Omega|^{\frac{1}{p}}\right) + ( k_1 \beta h +k_2)|\Omega|^{\frac{1}{p}}\right)\\
&&= k_1\beta c_1\|u(\cdot,t)\|_{L^{p}(\Omega)}^{p}+\left(k_1\beta
    c_2|\Omega|^{\frac{1}{p}}+(k_1\beta h+k_2)|\Omega|^{\frac{1}{p}}
\right)\|u(\cdot,t)\|_{L^{p}(\Omega)}^{p-1}.
\end{eqnarray*}
Thus, we have that
\begin{eqnarray*}
\frac{d}{dt}\|u(\cdot,t)\|_{L^{p}(\Omega)}^{p}\leq -p\|u(\cdot,t)\|_{L^{p}(\Omega)}^{p}
&+&pk_1\beta c_1\|u(\cdot,t)\|_{L^{p}(\Omega)}^{p}\\
&+&p\left[k_1\beta c_2|\Omega|^{\frac{1}{p}}+(k_1 \beta h+k_2)|\Omega|^{\frac{1}{p}}\right]\|u(\cdot,t)\|_{L^{p}(\Omega)}^{p-1}\\
&=&p\|u(\cdot,t)\|_{L^{p}(\Omega)}^{p}\left[-1+k_1\beta c_1+\frac{\left[k_1 \beta c_2+k_1\beta h+k_2\right]
|\Omega|^{\frac{1}{p}}}{\|u(\cdot,t)\|_{L^{p}(\Omega)}}\right].
\end{eqnarray*}

Letting $\varepsilon=1-k_1\beta c_1$, when
\[\|u(\cdot,t)\|_{L^{p}(\Omega)}\geq(1+\sigma)\frac{\left(k_1\beta
    c_2+k_1\beta h+k_2\right)|
\Omega|^{\frac{1}{p}}}{\varepsilon},\] we have that
\[
\frac{d}{dt}\|u(\cdot,t)\|_{L^{p}(\Omega)}^{p}\leq p\|u(\cdot,t)\|_{L^{p}(\Omega)}^{p}\left(-\varepsilon+\frac{\varepsilon}{1+\sigma}\right)
=-p\frac{\sigma}{1+\sigma}\varepsilon\|u(\cdot,t)\|_{L^{p}(\Omega)}^{p}.\]
Therefore when
$\|u(\cdot,t)\|_{L^{p}(\Omega)}\geq(1+\sigma)\frac{\left(k_1\beta
    c_2+k_1\beta h+k_2\right)|\Omega|^{\frac{1}{p}}}{\varepsilon}$, 
\[
\|u(\cdot,t)\|_{L^{p}(\Omega)}^{p}\leq e^{-\frac{\varepsilon\sigma p}{1+\sigma}t}\|u(\cdot,0)\|_{L^{p}(\Omega)}
\leq e^{-\frac{\sigma p(1-k_1\beta c_1)}{1+\sigma}t}\|u(\cdot,0)\|_{L^{p}(\Omega)}\]
what concludes the proof. 
\end{proof}

The next result generalizes Theorem 3.3 of \cite{DaSilvaPereira}, 
Theorem 3.3 of \cite{FAH} and Theorem 8 of \cite{Silva3}.

\begin{teo}\label{attractor}
In addition of the hypotheses assumed in Lemma \ref{absorbing}, 
suppose that  (\ref{dissipglinha}) holds and lets 
$\|J_x\|_{r} =\sup_{x\in \Omega}\frac{\partial }{\partial x} \|J(x,
\cdot)\|_{L^{r}(\Omega)}.$ Then there
exists a global attractor $\cal A$ for the flow $T(t)$ generated
by (\ref{CP}) in $L^{p}(\Omega)$, which is contained in the
 ball of radius $R$.
\end{teo}

\begin{proof} If $u(\cdot,t)$ is the solution of (\ref{CP}) with initial
condition $u(\cdot,0)$. For $x\in \Omega$  we have, by the  variation of constants
formula,
\begin{equation}
u(x,t)=e^{-t}u(x,0)+\intop_{0}^{t}e^{s-t}g(\beta Kf(u)(x,s)+ \beta h)ds.
\label{FVC}
\end{equation}
Consider
$$ T_{1}(t)u(x)=e^{-t}u(x,0)$$
and
$$T_{2}(t)u(x)=\intop_{0}^{t}e^{s-t}g(\beta Kf(u)(x,s)+ \beta h)ds.$$
Then, assuming that $u(\cdot,0)\in {\cal C}$, where ${\cal C}$ is a bounded set in
$L^{p}(\Omega)$, (for example $B(0,\rho)$), it follows that
$$\|T_{1}(t)u\|_{L^{2}}\underset{t\to\infty}{\longrightarrow} 0
~ \mbox{uniformly in} ~u.$$
Also, using (\ref{FVC}), we have that $\|u(\cdot,t)\|_{L^{p}(\Omega)}\leq
L$, for $t\geq 0$, where
$L=\max\left\{\rho, \frac{2\left(k_1\beta c_2+k_1\beta h+k_2\right)
|\Omega|^{\frac{1}{p}}}{1-k_1\beta c_1}\right\}$. 
Therefore, for $t\geq 0$, we have that
\begin{eqnarray*}
\frac{\partial T_{2}(t)u(x)}{\partial x}&=&
\intop_{0}^{t}e^{s-t}\frac{\partial}{\partial
x}g(\beta Kf(u)(x,s)+\beta h)ds\\
&=&\beta\intop_{0}^{t} e^{s-t} g'(\beta Kf(u)(x,s)+
\beta h)\frac{\partial Kf(u)}{\partial x}(x,s)ds.
\end{eqnarray*}
Thus, using (\ref{dissipglinha}) and (\ref{estimateLp}), we obtain
\begin{eqnarray*}
\left\|\frac{\partial T_{2}(t)u}{\partial x}\right\|_{L^{p}(\Omega)}&\leq&
\intop_{0}^{t}e^{s-t}\|g'(\beta Kf(u)(\cdot,s)+
\beta h)\beta\frac{\partial Kf(u)}{\partial x}(\cdot,s)\|_{L^{p}(\Omega)}ds\\
&\leq&\intop_{0}^{t}e^{s-t}[k_3\beta\|J\|_{1}\|f(u(\cdot,s))\|_{L^{p}(\Omega)}\\ 
&+& k_3\beta h+k_4]\beta\|J_x\|_{1}\|f(u(\cdot,s))\|_{L^{p}(\Omega)}ds\\
&\leq&\intop_{0}^{t}e^{s-t}[k_3\beta (c_1\|u(\cdot,s)\|_{L^{p}(\Omega)}+c_2|\Omega|^{\frac{1}{p}})\\
&+&k_3\beta h+k_4]\beta\|J_x\|_{1} (c_1\|u(\cdot,s)\|_{L^{p}(\Omega)}+c_2|\Omega|^{\frac{1}{p}})ds\\
&\leq &[k_3\beta  (c_1\|u(\cdot,s)\|_{L^{p}(\Omega)}+c_2|\Omega|^{\frac{1}{p}})\\
&+&k_3\beta h+k_4]\beta \|J_x\|_{1} (c_1\|u(\cdot,s)\|_{L^{p}(\Omega)}+c_2|\Omega|^{\frac{1}{p}})\\
&\leq&[k_3\beta  (c_1L+c_2|\Omega|^{\frac{1}{p}})+k_3\beta h+k_4]\beta\|J_x\|_{1}(c_1L +c_2|\Omega|^{\frac{1}{p}}).
\end{eqnarray*}

It follows that, for $t>0$ and any $u\in {\cal C}$, the value of
$\left\|\frac{\partial T_{2}(t)u}{\partial x}\right\|_{L^{p}(\Omega)}$ is bounded by
a constant (independent of $t$ and $u$). Thus, for all $u\in {\cal C}$,
we have that $T_{2}(t)u$ belongs to a ball of $W^{1,2}(\Omega)$.
From Sobolev's Imbedding Theorem, it follows that
$$\bigcup_{t\geq 0}T_{2}(t){\cal C}$$
is  relatively compact. Therefore, the result follows from Theorem
\ref{TeoremaTeman}, the attractor $\cal A$ being  the set
$\omega$-limit of the ball $B(0,R)$. 
\end{proof}

\section{Comparison and boundedness results} \label{comparison}

In this section  we prove  a comparison result that generalizes the Theorem
2.7 of \cite{Orland} (where $g\equiv\tanh, ~f(x)=x,\forall ~x\in
\mathbb{R}$ and $h=0$) and Theorem 4.2 of \cite{DaSilvaPereira} 
(where $f(x)=x, ~\forall ~x\in \mathbb{R}$).

\begin{defn}
A function $v(x,t)$ is a subsolution of the Cauchy problem for
(\ref{CP}) with initial condition $u(\cdot,0)$ if $v(x,0)\leq
u(x,0)$ for almost all $x\in\Omega, ~v$ is continuously
differentiable with respect to $t$ and satisfies
\begin{equation}\label{8.33}
\frac{\partial v(x,t)}{\partial t}\leq
-v(x,t)+g(\beta Kf(v)(x,t)+\beta h), 
\end{equation}
almost everywhere (a.e.).
\end{defn}

Analogously, a function $V(x,t)$ is a super solution if has the
same regularity properties as above, satisfies (\ref{8.33}) with
reversed inequality and $V(x,0)\geq u(x,0)$ for almost all $x\in
\Omega$.

\begin{teo}
In addition to the hypotheses of Theorem \ref{attractor}, assume 
that the functions $g$ and $f$ are monotonic and Lipschitz continuous on bounded  with 
Lipschitz's constants $N$ and $M$, respectively. Let 
$v(w,t), ~[V(w,t)]$ be a subsolution [super solution] of the
Cauchy problem of (\ref{CP}) with initial condition $u(\cdot,
0)$. Then
$$v(x,t)\leq u(x,t)\leq V(x,t), ~\mbox{a.e.}.$$
\end{teo}

\begin{proof}
Define the operator $G$ on $L^{\infty}(\Omega \times [0,T])$ by
$$G(w)(x,t)=e^{-t}w(x,0)+\intop_{0}^{t}e^{-(t-s)}g(\beta(Kf(w)(x,s)+h))ds.$$
Then  $(G(w))(x,0)=w(x,0)$. Also, since $f$ and $g$ are monotonic, 
it follows that $G$ is monotonic, that is, for any $w_{1},w_{2}\in
L^{\infty}(\Omega\times[0,T])$ with $w_{1} \geq w_{2}$ (a.e. in
$\Omega\times [0,T]),$ we have $G(w_{1})\geq G(w_{2})$ (a.e. in
$\Omega\times [0,T]$).

From (\ref{estimateLq}), we obtain
\begin{eqnarray*}
|G(w)(x,t)|&\leq& e^{-t}|w(x,0)|+\intop_{0}^{t}e^{-(t-s)}|g(\beta
Kf(w)(x,s)+\beta h)|ds\\
&\leq&e^{-t}|w(x,0)|+\intop_{0}^{t}e^{-(t-s)}[k_{1}|\beta
Kf(w)(x,s)+\beta h|+k_{2}]ds\\
&\leq&e^{-t}|w(x,0)|+\intop_{0}^{t}e^{-(t-s)}k_{1}\beta
|Kf(w)(x,s)|ds+\intop_{0}^{t}e^{-(t-s)}(k_{1}\beta h+k_{2})ds.
\end{eqnarray*}
Since  $|Kf(w)(x,s)|\leq\|J\|_{1}\|f(w)\|_{\infty}\leq
k_1\|w\|_{\infty}+k_2$ a.e. in $\Omega\times [0,T]$, we obtain
\begin{eqnarray*}
\|G(w)\|_{\infty}&\leq& e^{-t}\|w(\cdot,0)\|_{\infty}+\intop_{0}^{t}e^{-(t-s)}k_{1}\beta
(k_1\|w\|_{\infty}+k_2)ds+\intop_{0}^{t}e^{-(t-s)}(k_{1}\beta h+k_{2})ds\\
&\leq&\|w\|_{\infty}+k_{1}\beta (k_1\|w\|_{\infty}+k_2) +(k_{1}\beta h+k_{2}).
\end{eqnarray*}
Therefore $G:L^{\infty}(\Omega\times [0,T])\to
L^{\infty}(\Omega\times [0,T])$.

Furthermore, if  $\beta NMT< 1, ~G$ is a contraction in any
subset of functions of $L^{\infty}(\Omega\times [0,T])$ with the
same values at $t=0$. In fact
\begin{eqnarray*}
|G(w_{1})(x,t)-G(w_{2})(x,t)|&=&\left|\intop_{0}^{t}e^{-(t-s)}[g(\beta(Kf(w_{1})(x,s)+\beta
h)-g(\beta(Kf(w_{2})(x,s)+\beta h)]ds\right|\\
&\leq&
\intop_{0}^{t}e^{-(t-s)}N\beta|Kf(w_{1})(x,s)-Kf(w_{2})(x,s)|ds\\
&\leq &
\intop_{0}^{t}e^{-(t-s)}N\beta(K|f(w_1)-Kf(w_2)|(x,s))ds\\
& \leq &
\intop_{0}^{t}e^{-(t-s)}N\beta K\|f(w_{1})-f(w_{2})\|_{\infty}ds \\
&= & N\beta T
\|f(w_{1})-f(w_{2})\|_{\infty}\intop_{0}^{t}e^{-(t-s)}ds\\
&\leq& N\beta MT\|w_{1}-w_{2}\|_{\infty},
\end{eqnarray*}
a.e. in $\Omega \times [0,T]$. Hence
$\|G(w_{1})-G(w_{2})\|_{\infty}\leq\beta NMT
\|w_{1}-w_{2}\|_{\infty}$. Therefore, if $\beta NMT<1, ~G$ is
a contraction. Thus, if $u(x,t)$ is a solution of (\ref{CP}) with
$u^{0}=u(x,0)$, we have
$$u=\lim_{n\to \infty} G^{n}(u^{0})$$
on $L^{\infty}(\Omega\times[0,T])$. The same holds for a solution
$\widetilde{u}$ with $\widetilde{u}^{0}=\widetilde{u}(x,0)$. If
$\widetilde{u}^{0}\leq u^{0}$  a.e., with  $g$ and $f$ monotonic, it
follows that
$$G^{n}(\widetilde{u}^{0})\leq G^{n}(u^{0}), ~\mbox{a.e.}$$
Now, if $v$ is a subsolution of (\ref{CP}), 
it's easy to see that
$$v(x,t)\leq
e^{-t}v(x,0)+\intop_{0}^{t}e^{-(t-s)}g(\beta(Kf(v)(x,s)+h))ds, ~\mbox{a.e.}$$
Therefore $v(x,t)\leq G(v)(x,t)$, a.e., and
since  $g$ and $f$ are monotonic, it follows that $v(w,t)\leq
G^{n}(v)(x,t)$ a.e. 
Thus, $v(x,t)\leq z(x,t),~ \mbox{a.e.}$, where
\[z=\lim_{n\to \infty}G^{n+1}(v).\]
Now, from the continuity of $G$, it follows that
\[G(z)=G\left(\lim_{n \to \infty}
G^{n}(v)\right)=\lim_{n \to \infty} G^{n+1}(v)=z.\]
Therefore $z$ is a fixed point of $G$, that is, $z$  is a
solution of (\ref{CP}) in $\Omega\times[0,T]$ with initial
condition $z(\cdot,0)=v(\cdot,0)$. Thus, if $z(\cdot,0)\leq
u(\cdot,0), ~\mbox{a.e.}$, then
\[v\leq z \leq u, ~\mbox{a.e. in} ~\Omega\times [0,T],\]
where $u$ is the solution of (\ref{CP}) with initial condition
$u(\cdot, 0)$.
If  $V(x,t)$ is a super solution,  we obtain, by the same arguments
\[u\leq\widetilde{z}\leq V, ~\mbox{a.e. ~in} ~\Omega\times [0,T].\]
Therefore
\[v(x,t)\leq u(x,t)\leq V(x,t), ~\mbox{a.e.}\]
in $\Omega\times[0,T]$.

Since the estimates above do not depend on the initial condition,
we may extend the result to $[T, 2T]$ and, by iteration,
we can complete the proof of the theorem. 
\end{proof}

\begin{Remark} If we add the hypothesis $g(x)<\rho$,
the comparison result holds in the ball
$\mathbb{B}=\{L^{\infty}(\Omega\times [0,T]),
\|\cdot\|_{\infty}\leq\rho\}$.
\end{Remark}

In fact, it is enough to prove that
$G|_{\mathbb{B}}:\mathbb{B}\rightarrow \mathbb{B}$. But
\[|(G|_{\mathbb{B}}(w))(x,t)|\leq e^{-t}|w(x,0)|+\rho\intop_{0}^{t}e^{-(t-s)}ds.\]
Hence
\[\|(G|_{\mathbb{B}}(w))\|_{\infty}\leq e^{-t}\|w\|_{\infty}+\rho
\intop_{0}^{t}
e^{-(t-s)}ds\leq \rho e^{-t}+ \rho \intop_{0}^{t}e^{-(t-s)}ds=\rho.\]
Therefore, $G|_{\mathbb{B}}(w)\in \mathbb{B}$.

\begin{teo}\label{teolimit}
In the same conditions from Theorem \ref{attractor}, we have that
the attractor $\cal A$ belongs to the ball $\|\cdot\|_{\infty}\leq
\rho$ in $L^{\infty}(\Omega)$, where
$\rho=k_1\beta\|J\|_{q}c_1 R+k_1\beta\|J\|_{q}c_2
|\Omega|^{\frac{1}{p}}+k_1\beta h+k_2.$ 
\end{teo}

\begin{proof}  From Theorem \ref{attractor} the attractor is
contained in the ball $B[0,\rho]$ in $L^{p}(\Omega)$.

Let $u(x,t)$ be a solution of (\ref{CP}) in ${\cal A}$. Then, for $x\in\Omega$, by
the variation of constants formula
$$
u(x,t)=e^{-(t-t_{0})}u(x,t_{0})+\intop_{t_{0}}^{t}e^{-(t-s)}g(\beta Kf(u)(x,s)+\beta
h)ds.
$$
Since $\|u(\cdot,t)\|_{L^{p}(\Omega)}\leq R$ for all $u\in {\cal A}$,
we obtain
 for all $(x,t) \in \Omega \times \mathbb{R}^{+}$
letting $t_{0}\to -\infty$
$$
u(x,t)=\intop_{-\infty}^{t}e^{-(t-s)}g(\beta Kf(u)(x,s)+\beta h)ds,
$$
where the equality above is in the sense of $L^{p}(\Omega)$. Thus,
using (\ref{dissipg}), we have
\begin{eqnarray*}
|u(x,t)|&\leq&
\intop_{-\infty}^{t}e^{-(t-s)}|g(\beta Kf(u)(x,s)+\beta
h)|ds\\
&\leq&\intop_{-\infty}^{t}e^{-(t-s)}[k_1\beta |Kf(u(x,t))+\beta h|+ k_2]ds\\
&\leq&\intop_{-\infty}^{t}e^{-(t-s)}[k_1\beta\|J\|_{q}\|f(u(\cdot,t))\|_{L^{p}(\Omega)}+k_1\beta h+k_2]ds\\
&\leq&\intop_{-\infty}^{t}e^{-(t-s)}[k_1\beta\|J\|_{q}(c_1\|u(\cdot,t)\|_{L^{p}(\Omega)}+c_2|\Omega|^{\frac{1}{p}})+k_1\beta h+ k_2]ds\\
&\leq&\intop_{-\infty}^{t}e^{-(t-s)}[k_1\beta\|J\|_{q}(c_1 R+c_2|\Omega|^{\frac{1}{p}})+k_1\beta h+k_2]ds\\
&\leq&\intop_{-\infty}^{t}\rho e^{-(t-s)} ds.
\end{eqnarray*}
Therefore $\|u(\cdot,t)\|_{\infty}\leq\rho$, as claimed
\end{proof}

\section{Existence of a Lyapunov's functional}\label{lyap}

In this section we exhibit a continuous ``Lyapunov's functional''
for the flow of (\ref{CP}), restricted to the ball of radius $\rho$
in $L^{\infty}(\Omega)$, concluding that this flow is gradient,
in the sense of \cite{Hale}. 

Initially, we claim that $\{L^{p}(\Omega),\|\cdot\|_{\infty}\leq\rho\}$ is an
invariant set for the flow generated by (\ref{CP}).

In fact, let 
\[
u(x,t)=e^{-t}u(x,0)+\intop_{0}^{t}e^{-(t-s)}g(\beta Kf(u(x,s))+\beta h)ds
\]
be the solution of (\ref{CP}) with initial condition $u(\cdot,0)\in\{L^{2}(\Omega),\|\cdot\|_{\infty}\leq\rho\}.$ Then
\begin{eqnarray*}
|u(x,t)|&\leq&e^{-t}|u(x,0)|+\intop_{0}^{t}e^{-(t-s)}|g(\beta Kf(u(x,s))+\beta h)|ds\\
&\leq&e^{-t}|u(x,0)|+\intop_{0}^{t}e^{-(t-s)}[k_1\beta |Kf(u(x,t))+\beta h|+k_2]ds\\
&\leq&e^{-t}|u(x,0)|+\intop_{0}^{t}e^{-(t-s)}[k_1\beta\|J\|_{q}\|f(u(\cdot,t))\|_{L^{p}(\Omega)}+k_1\beta h+k_2]ds\\
&\leq&e^{-t}|u(x,0)|+\intop_{0}^{t}e^{-(t-s)}[k_1\beta\|J\|_{q}(c_1\|u(\cdot,t)\|_{L^{p}(\Omega)}+c_2|\Omega|^{\frac{1}{p}})+k_1\beta h+k_2]ds\\
&\leq&e^{-t}|u(x,0)|+\intop_{0}^{t}e^{-(t-s)}\rho ds.
\end{eqnarray*}
Whence,
\begin{eqnarray*}
\|u(\cdot,t)\|_{\infty}&\leq&e^{-t}\|u(\cdot,0)\|_{\infty}+\rho \intop_{0}^{t}e^{-(t-s)}ds\\
&\leq&e^{-t}\rho+\rho\intop_{0}^{t}e^{-(t-s)}ds\\
&=&\rho.
\end{eqnarray*}

For to exhibit a continuous ``Lyapunov's functional''
for the flow of (\ref{CP}),  we assume that the 
functions $f$ and $g$ satisfy the following conditions:
\begin{equation}\label{limitationg} 
0<|g(x)|<\rho, ~ \forall ~ x\in\mathbb{R},
\end{equation}
the function $g^{-1}$ is continuous in $]-\rho,\rho[$ and the
function 
\begin{equation}\label{energy}
\theta(m)=-\frac{1}{2}f(m)^{2}-hf(m)-\beta^{-1}i(m), ~ m\in [-\rho,\rho],
\end{equation}
where $i$ is defined by
$$i(m)=-\intop_{0}^{f(m)}g^{-1}(f^{-1}(s))ds, ~ m\in [-\rho,\rho],$$
has a global minimum $\bar{m}$ in $]-\rho,\rho[$.

Note that if (\ref{limitationg}) holds, it follows that
(\ref{dissipg}) holds with $k_{1}=0$ and
$k_{2}=\rho$. 

Motivated by functionals that appear in
\cite{DaSilvaPereira,Silva4,DaSilvaPereira2, Kubota} and \cite{Orland}, 
we define the functional ${\cal F}:\{L^{p}(\Omega),~\|u
\|_{\infty}\leq\rho\}\to \mathbb{R}$ by
\begin{equation}\label{4.8}
{\cal F}(u)=\intop_{\Omega}[\theta(u(x))-\theta(\overline{m})]dx+
\frac{1}{4}\intop_{\Omega}\intop_{\Omega}J(x,y)[f(u(x))-f(u(y))]^{2}dxdy, 
\end{equation}
where $\theta$ is given in (\ref{energy}), which has been adapted from functions considered in \cite{Orland} 
and \cite{DaSilvaPereira}.

Note that the functional in (\ref{4.8}) is defined
in the whole space $\{L^{p}(\Omega), \,\, \|u\|_{\infty}\leq\rho\}$. 
Furthermore, using the hypotheses on $f$ and $g$ and Lebesgue's 
Dominated Convergence Theorem, we obtain the following result:

\begin{teo}\label{Teorema 4.1}
In addition to the hypotheses of Theorem \ref{attractor}, assume 
that the hypotheses established in (\ref{limitationg}) and
(\ref{energy}) hold. Then the functional given in
(\ref{4.8}) is continuous in the topology of
$L^{p}(\Omega).$
\end{teo}

Now we are ready to prove the main result of this section.

\begin{teo}\label{Teorema 4.2}
In addition of the hypotheses from Theorem \ref{attractor}, 
assume that the hypotheses established in (\ref{limitationg}) 
and (\ref{energy}) hold and that $f$ has positive derivative. 
Let $u(\cdot , t)$ be a solution of (\ref{CP})
with $\|u(\cdot,t)\|_{\infty}\leq \rho$. Then ${\cal F}(u(\cdot,t))$ is
differentiable with respect to $t$ for $t>0$ and
\[\frac{d}{dt}{\cal F}(u(\cdot,t))=-{\cal I}(u(\cdot , t))\leq 0,\]
where, for any $u\in L^{p}(\Omega)$ with $\|u\|_{\infty}\leq \rho$,
$${\cal I}(u(\cdot))=\intop_{\Omega}[K(f(u)(x))+h-\beta^{-1}g^{-1}(u(x))]
[g(\beta K(f(u)(x))+\beta h)-u(x)]f'(u(x))dx.$$
Furthermore, the integrand in ${\cal I}(u(\cdot))$ is a non negative
function and, $u$ is a critical point of ${\cal F}$ if only if
$u$ is an equilibrium of (\ref{CP}).
\end{teo}

\begin{proof} From hypotheses on $g$ and $f$, it follows that ${\cal F}(u(\cdot,t))$
is well defined for all $t\geq 0$. We assume first that, given
$t>0$, there exists $\varepsilon
>0$ such that $\|u(\cdot,s)\|_{\infty}\leq \rho-\varepsilon,$ for $s\in
\Delta$ where $\Delta$ is a closed finite interval containing
$t$. For $s\in\Delta$ we write
$${\cal F}(u(\cdot,s))=\intop_{\Omega}\phi(x,s)dx ~\mbox{and}
~ {\cal I}(u(\cdot,s))=\intop_{\Omega}\iota(x,s)dx.$$
As
\begin{eqnarray*}
\frac{\partial\phi}{\partial
s}(x,s)&=&[-f(u(x,s))-h+\beta^{-1}g^{-1}(u(x,s))]f'(u(x,s))\frac{\partial}{\partial s}u(x,s)\\
&&+\frac{1}{2}\intop_{\Omega}J(x,y)[f(u(x,s))-f(u(y,s))]\left[f'(u(x,s))\frac{\partial u(x,s)}{\partial
s}-f'(u(y,s))\frac{\partial u(y,s)}{\partial s}\right]dy,
\end{eqnarray*}
the hypotheses on $g$, $f$ and $f'$ imply that
$\frac{\partial \phi(x,s)}{\partial s}$ is almost
everywhere continuous and bounded in $x$ for $s\in\Delta$. Thus
$$\sup_{s\in \Delta}\left\|\frac{\partial \phi(\cdot,s)}{\partial
s}\right\|_{L^{1}}<\infty.$$
Therefore, we can derive under the integration sign obtaining
\begin{eqnarray*}
&&\frac{d}{ds}{\cal F}(u(\cdot,s))=\intop_{\Omega}[-f(u(x,s))-h+\beta^{-1}g^{-1}(u(x,s))]f'(u(x,s))\frac{\partial
u(x,s)}{\partial s}dx\\
&&+\frac{1}{2}\intop_{\Omega}\intop_{\Omega}J(x,y)[f(u(x,s))-f(u(y,s))]\Bigg[f'(u(x,s))\frac{\partial u(x,s)}{\partial s}-f'(u(y,s))\frac{\partial u(y,s)}{\partial s}\Bigg]dxdy.
\end{eqnarray*}
But
\begin{eqnarray*}
&&\intop_{\Omega}\intop_{\Omega}J(x,y)[f(u(x,s))-f(u(y,s))]
\left[f'(u(x,s))\frac{\partial u(x,s)}{\partial s}- f'(u(y,s))\frac{\partial
u(y,s)}{\partial s}\right]dxdy\\
&&=\intop_{\Omega}\intop_{\Omega}J(x,y)f(u(x,s)) f'(u(x,s))\frac{\partial
u(x,s)}{\partial s}dxdy \\
&&- \intop_{\Omega}\intop_{\Omega}J(x,y) f(u(x,s)) f'(u(y,s)) \frac{\partial u(y,s)}{\partial s}dxdy\\
&&-\intop_{\Omega}\intop_{\Omega}J(x,y)f(u(y,s)) f'(u(x,s))\frac{\partial
u(x,s)}{\partial s}dxdy\\
&&+ \intop_{\Omega}\intop_{\Omega}J(x,y)f(u(y,s)) f'(u(y,s)) \frac{\partial u(y,s)}{\partial s}dxdy\\
&&=2\intop_{\Omega}\intop_{\Omega}J(x,y) f(u(x,s)) f'(u(x,s)) \frac{\partial
u(x,s)}{\partial s}dxdy \\
&&- 2\intop_{\Omega}\intop_{\Omega}J(x,y) f(u(y,s)) f'(u(x,s))\frac{\partial u(x,s)}{\partial s}dxdy\\
&&=2\intop_{\Omega}\left(\intop_{\Omega}J(x,y)dy \right) f(u(x,s)) f'(u(x,s)) \frac{\partial
u(x,s)}{\partial s}dx\\
&&-2\intop_{\Omega} \left(\intop_{\Omega}J(x,y) f(u(y,s)) dy \right) f'(u(x,s))\frac{\partial u(x,s)}{\partial s}dx.
\end{eqnarray*}
Using the fact that 
$$\intop_{\Omega}J(x,y)dy = \intop_{\Omega}J(x,y)dx =1,$$
it follows that
\begin{eqnarray*}
\frac{d}{ds}{\cal F}(u(\cdot ,
s))&=&\intop_{\Omega}\bigg[-f(u(x,s))-h+\beta^{-1}g^{-1}(u(x,s))\bigg] f'(u(x,s))\frac{\partial
u(x,s)}{\partial s}dx\\
&&+\intop_{\Omega}[f(u(x,s))-Kf(u(x,s))]f'(u(x,s))\frac{\partial
u(x,s)}{\partial s} dx\\
&&=\intop_{\Omega}\bigg[-f(u(x,s))-h+\beta^{-1}g^{-1}(u(x,s))+f(u(x,s))\\
&& -~ Kf(u(x,s))\bigg] f'(u(x,s))\frac{\partial
u(x,s)}{\partial s}dx\\
&&=-\intop_{\Omega}\bigg[Kf(u(x,s))+h-\beta^{-1}g^{-1}(u(x,s))\bigg]f'(u(x,s))\frac{\partial
u(x,s)}{\partial s} dx\\
&&=-\intop_{\Omega}\bigg[Kf(u(x,s))+h-\beta^{-1}g^{-1}(u(x,s))\bigg]\big[-u(x,s)\\
&&+~g(\beta Kf(u(x,s))+\beta h)\big] f'(u(x,s)) dx\\
&&=-{\cal I}(u(\cdot,s)).
\end{eqnarray*}

This proves the first part of theorem with the additional hypothesis that
$\|u(\cdot,s)\|_{\infty}\leq\rho-\varepsilon$, for $s\in\Delta$ and
some $\varepsilon>0$, where $\Delta$ is a closed
finite interval containing $t$. 

We claim that this hypothesis actually
holds for all $t>0$. In fact, let $\lambda(x,t)$ be the solution of (\ref{CP}) such that
$\lambda(x,0)=\rho$ for any $x\in\Omega$. Then
$\lambda(x,t)=\lambda(t)$, 
where
$$\frac{d \lambda}{dt}=-\lambda(t)+g(\beta(\lambda(t)+h)).$$
Since $|g(x)|<\rho, ~\forall ~x\in\mathbb{R}$, it follows easily that 
$\lambda(t)<\rho$ for any $t>0$. As $u(x,0)\leq\rho$, we obtain 
by the Comparison Theorem
$$u(x,t)\leq\lambda(t)<\rho,$$
for almost every $x\in\Omega$ and $t>0$. Repeating the same
argument, starting from inequality $u(x,0)\geq -\rho,$ for almost
every $x\in\Omega$, we obtain $u(x,t)\geq -\lambda(t)>-\rho$, and
thus
$$\|u(\cdot,t)\|_{\infty}\leq\lambda(t)<\rho, ~\forall ~ t>0$$
and the claim follows by continuity.

To conclude the proof, it is enough to show that $u$ is a critical point
of ${\cal F}$ if and only if $u$ is an equilibrium of (\ref{CP}).
For this, let $u(x)$ be a critical point of the functional
${\cal F}$, then ${\cal I}(u(\cdot))=0$. Since the integrand is non
negative almost everywhere, it  follows that
$$[(Kf(u)(x))+h-\beta^{-1}g^{-1}(u(x))]f'(u(x))[g(\beta(Kf(u)(x)+h))-u(x)]=0$$
almost everywhere. Since $f'(u(x))>0$, for all $x\in \mathbb{R}$, we
have that 
$$[(Kf(u)(x))+h-\beta^{-1}g^{-1}(u(x))][g(\beta(Kf(u)(x)+h))-u(x)]=0$$
almost everywhere. But the annihilation of any of these factors
implies that $$g(\beta Kf(u)(x)+\beta h)=u(x).$$
Reciprocally, if $u$ is a equilibrium of (\ref{CP}), it is easy
to see that ${\cal I}(u(\cdot))=0$. 
\end{proof}

As a immediate consequence of the existence of the functional
${\cal F}$, we obtain the following result.

\begin{coro}\label{Coro 4.4}
Under the same hypotheses of Theorem \ref{Teorema 4.2}, there are no non trivial recurrent points under the flow of (\ref{CP}). 
\end{coro}

\begin{Remark}\label{Remark 4} The integrand in the functional ${\cal F}$ above
is always  non negative since $J$ is positive and
 $\overline{m}$ is a  global minim of $\theta$. 
Thus, ${\cal F}$ is lower bounded.
\end{Remark}

We recall that a $C^{r}$-semigroup, $T(t)$, is gradient if
each bounded positive orbit is precompact and there exists a
Lyapunov's Functional for $T(t)$ (see \cite{Hale}).

\begin{prop}\label{gradient} Assume the same hypotheses of Theorem \ref{Teorema 4.2}. 
Then the flow generated by equation (\ref{CP}) is gradient.
\end{prop}
\begin{proof} The precompacity of the orbits follows from the existence of
the global attractor (see Theorem \ref{attractor}).
From Theorems \ref{Teorema 4.1} and \ref{Teorema
4.2}, and Remark \ref{Remark 4}, we have existence of a continuous
Lyapunov's functional. 
\end{proof}

From Proposition \ref{gradient}, we have the following characterization of the attractor (see \cite{Hale} -
Theorem 3.8.5).
\begin{teo}\label{Teorema 5.2} Assume the same assumptions of Proposition \ref{gradient}. Then the attractor $\cal A$ is the unstable set of the 
equilibrium point set of $T(t)$, that is, ${\cal A} =W^{u}(E)$.
\end{teo}

%


\end{document}